\documentclass[11pt, a4paper]{amsart}
\usepackage{amssymb, array, amsmath, amscd, pdfpages, enumerate, amsthm, setspace, hyperref,mathtools}

\usepackage[utf8]{inputenc}
\usepackage{amssymb}
\usepackage{bm}
\usepackage{graphicx}

\newcommand{\yref}{y_{\mbox{\scriptsize\textit{ref}}}}
\newcommand{\wdist}{w_{\mbox{\scriptsize\textit{dist}}}}

\newcommand{\citel}[2]{\cite[#2]{#1}}

\newcommand{\Acomp}{A_s}
\newcommand{\Bcomp}{B_s}

\newcommand{\Qcomp}{Q_s}

\newenvironment{ORP}{\textbf{The Output Regulation Problem.}\it}{}

\newtheorem{theorem}{Theorem}[section]

\newtheorem{lemma}[theorem]{Lemma}

\theoremstyle{definition}

\newcommand{\pmat}[1]{\begin{bmatrix}#1 \end{bmatrix}}
\newcommand{\pmatsmall}[1]{\begin{bsmallmatrix}#1\end{bsmallmatrix}}

\newcommand{\eps}{\varepsilon}

\DeclareMathOperator{\re}{Re}
\DeclareMathOperator{\im}{Im}
\DeclareMathOperator{\Span}{span}
\DeclareMathOperator{\diag}{diag}

\newcommand{\T}{\mathbb{T}}

\newcommand*{\C}{{\mathbb{C}}}     
\newcommand*{\R}{{\mathbb{R}}}

\newcommand*{\N}{{\mathbb{N}}}

\newcommand*{\Lin}{{\mathcal{L}}}

\newcommand*{\abs} [1]{\lvert#1\rvert}
\newcommand*{\norm}[1]{\lVert#1\rVert}
\newcommand*{\set} [1]{\{#1\}}
\newcommand*{\setm}[2]{\{\,#1\mid#2\,\}}   
\newcommand*{\iprod}[2]{\langle#1,#2\rangle}

\newcommand*{\Lp}[1][p]{L^{#1}}

\newcommand{\Abs}[2][default]{\ifthenelse{\equal{#1}{default}}{\left\lvert#2\right\rvert}{\ldelim{#1}{\lvert}#2\rdelim{#1}{\rvert}}}
\newcommand{\Norm}[2][default]{\ifthenelse{\equal{#1}{default}}{\left\lVert#2\right\rVert}{\ldelim{#1}{\lVert}#2\rdelim{#1}{\rVert}}}

\newcommand*{\Iprod}[3][default]{\ifthenelse{\equal{#1}{default}}{\left\langle#2,#3\right\rangle}{\ldelim{#1}{\langle}#2,#3\rdelim{#1}{\rangle}}}
\newcommand*{\Dualpair}[3][default]{\ifthenelse{\equal{#1}{default}}{\left\langle#2,#3\right\rangle}{\ldelim{#1}{\langle}#2,#3\rdelim{#1}{\rangle}}}

\newcommand*{\List}[2][1]{\set{#1,\ldots,#2}}

\newcommand{\eq}[1]{\begin{align*}#1\end{align*}}
\newcommand{\eqn}[1]{\begin{align}#1\end{align}}

\newcommand{\ga}{\alpha}
\newcommand{\gb}{\beta}

\newcommand{\gl}{\lambda}
\newcommand{\gw}{\omega}

\newcommand{\inv}{^{-1}}

\newcommand*{\ddb}[2][1]{\ifthenelse{\equal{#1}{1}}{\frac{d}{d#2}}{\frac{d^{#1}}{d#2^{#1}}}}
\newcommand*{\pd}[3][1]{\ifthenelse{\equal{#1}{1}}{\frac{\partial{#2}}{\partial{#3}}}{\frac{\partial^{#1}{#2}}{\partial#3^{#1}}}}

\newcommand*{\keyterm}[1]{\emph{#1}}

\begin{document}

\title[Output Tracking of a Kelvin--Voigt Beam]{A Reduced Order Controller for Output Tracking of a Kelvin--Voigt Beam}

\thispagestyle{plain}

\author[L. Paunonen]{Lassi Paunonen}
\address[L. Paunonen]{Mathematics, Faculty of Information Technology and Communication Sciences, Tampere University, PO.\ Box 692, 33101 Tampere, Finland}
\email{lassi.paunonen@tuni.fi}

\author[D. Phan]{Duy Phan}
\address[D. Phan]{Institut f\"{u}r Mathematik, Universit\"at Innsbruck,
Technikerstra{\ss}e 13, 6020 Innsbruck, Austria}
\email{duy.phan-duc@uibk.ac.at}

\thanks{This work was supported by the Academy of Finland Grants number 298182 and 310489.}

\begin{abstract}
We study output tracking and disturbance rejection for an Euler-Bernoulli beam with Kelvin-Voigt damping. The system has distributed control and pointwise observation. As our main result we design a finite-dimensional low-order internal model based controller that is based on a spectral Galerkin method and model reduction by Balanced Truncation. The performance of the designed controller is demonstrated with numerical simulations and compared to the performance of a low-gain internal model based controller.
\end{abstract}


\maketitle

\section{Introduction}

In this paper we design a finite-dimensional error feedback controller for output tracking and disturbance rejection of an Euler--Bernoulli beam equation with Kelvin-Voigt damping on $\Omega=(-1,1)$~\citel{ItoMor98}{Sec.~3},
\begin{subequations}
    \label{eq:EBmodel}
  \eqn{
\MoveEqLeft[5] v_{tt}(\xi,t)
+\left( EI v_{\xi\xi}+d_{KV}I v_{\xi\xi t} \right)_{\xi\xi}(\xi,t)  + d_v v_t(\xi,t)\\
& = 
b_1(\xi)u_1(t)+b_2(\xi)u_2(t)
+ B_{d0}(\xi) \wdist(t) 
    \\
     v(-1,t) &= v_\xi(-1,t)= 0,
    \\
    v(1,t) &=v_\xi(1,t)=0\\
     v(\xi,0) &=v_0(\xi), \qquad v_t(\xi,0)=v_1(\xi),\\
     y(t) &= (v(\xi_1,t),v(\xi_2,t))^T.
  }
\end{subequations}
The parameters $E>0$ and $I>0$ are the (constant) elastic modulus and 
the second moment of area, respectively,
and $d_{KV}>0$ and $d_v\geq 0$ are the coefficient associated to the Kelvin--Voigt damping and viscous damping, respectively.
The beam is assumed to have constant density $\rho=1$.
The system has two inputs $u(t)=(u_1(t),u_2(t))^T$ and two outputs $y(t)=(y_1(t),y_2(t))^T$ (see Section~\ref{sec:ORP} for details).

We study \keyterm{output tracking} and \keyterm{disturbance rejection},
where the aim is to design a control law in such a way that 
the output $y(t)$ of~\eqref{eq:EBmodel} converges to a given reference signal $\yref(t)$, i.e., $\norm{y(t)-\yref(t)}\to 0$ as $t\to\infty$, despite the external disturbance signals $\wdist(t)$.
The considered signals $\yref:[0,\infty)\to \R^p$ and
$\wdist:[0,\infty)\to \R^{n_d}$ are of the form
\begin{subequations}
  \label{eq:yrefwdist}
  \eqn{
    \hspace{-1ex}  \yref(t) &= a_0^1 + \sum_{k=1}^q (a_k^1 \cos(\gw_k t) + b_k^1 \sin(\gw_k t))\\
  \label{eq:yrefwdistwd}
    \hspace{-1ex} \wdist(t) &= a_0^2 + \sum_{k=1}^q (a_k^2 \cos(\gw_k t) + b_k^2 \sin(\gw_k t))
  }
\end{subequations}
for some known frequencies $\set{\gw_k}_{k=0}^q\subset \R$ with $0=\gw_0<\gw_1<\ldots<\gw_q$ and possibly \keyterm{unknown} constants $\set{a_k^j}_{k,j}\subset \R$ and $\set{b_k^j}_{k,j}\subset \R$ (any of the constants are allowed to be zero). 
This control problem --- typically called \keyterm{output regulation} --- has been studied extensively in the literature for controlled partial differential equations~\cite{XuSal14,Deu15,XuDubCPDE16,JinGuo19}
and distributed parameter systems~\cite{Poh81a,HamPoh00,ByrLau00,RebWei03,Imm06a,Imm07a,NatGil14,Pau16a}.

In this paper we solve the output tracking and disturbance rejection problem with a finite-dimensional dynamic error feedback controller introduced recently in~\cite{PauPha20}. The controller design is based on Galerkin approximation theory for a class of linear systems~\cite{BanKun84,Mor94}, in particular including controlled parabolic PDEs, and it uses model reduction to
reduce the dimension of the controller\footnote{Note that the Internal Model Principle~\cite{PauPoh10} requires a controller solving the control problem for all $\yref(t)$ and $\wdist(t)$ in~\eqref{eq:yrefwdist} necessarily has dimension of at least ``number of outputs $\times$ number of (complex) frequencies'', i.e., $2(2q+1)$. Thus in our setting the controller having ``low order'' means that the dimension is not much higher than $2(2q+1)$.}.
Output regulation of an Euler--Bernoulli beam with Kelvin--Voigt damping (single-input-single-output with clamped--free boundary conditions) was also considered in~\citel{PauPha20}{Sec.~V.C}, where the controller was based on the Finite Element Method. As the main novelty of this paper we instead base our controller on a \keyterm{spectral Galerkin method} with non-local basis functions based on Chebyshev polynomials~\cite{She95,SheTan11book}.

The motivation for this study comes from the fact that spectral methods are powerful numerical approximation  tools --- typically achieving great accuracy with low numbers of basis functions --- but to our knowledge they have not been used previously in controller design for output regulation of PDEs.
In addition, in the case of Chebyshev functions
the \textbf{Chebfun} MATLAB library (available at \url{https://www.chebfun.org/})~\cite{DriHal14book,Tre13book} can be employed in computing
the Galerkin approximation.

Finally, we compare the performance of our controller to the so-called ``simple'' internal model based controller, and demonstrate that our new controller achieves a considerably improved rate of convergence of the output tracking.

\subsection*{Acknowledgement}

The authors are very grateful to the anonymous reviewer who 
pointed out the reference~\cite{Mad90} and 
provided a detailed 
 proof of Lemma~\ref{lem:FormBddCoercive}. The boundedness and coercivity of $a(\cdot,\cdot)$ were left as open questions in the original version of the manuscript and the reviewer helped us very much in completing this work. 
The authors also thank Konsta Huhtala for helpful discussions and for carefully examining the manuscript.

\section{The Output Regulation Problem}
\label{sec:ORP}

In~\eqref{eq:EBmodel} the functions
$b_1(\cdot),b_2(\cdot)\in C^1(-1,1;\R)$
are fixed input profiles satisfying $b_j(\pm 1)=b_j'(\pm 1)=0$, $j=1,2$, 
and the disturbance input term has the form $B_{d0}(\xi)\wdist(t) = \sum_{k=1}^{n_d}b_{dk}(\xi)\wdist^k(t)$ for $\wdist(t)=(\wdist^k(t))_{k=1}^{n_d}$ and for some fixed but unknown profile functions $b_{dk}(\cdot)\in C^1(-1,1;\R)$ with $b_{dk}(\pm 1)=b_{dk}'(\pm 1)=0$ for $k\in \List{n_d}$\footnote{Theory would allow weaker assumptions on $b_j$ and $b_{dk}$, but spectral methods do not work well for discontinuous functions.}.

\noindent\begin{ORP}
  Design a dynamic error feedback controller such that the following hold.
  \begin{itemize}
    \item[\textup{(1)}] The closed-loop system is exponentially stable.
    \item[\textup{(2)}] For any initial values of the system~\eqref{eq:EBmodel} and the controller and for any $\set{a_k^j}_{k,j}\subset \R$ and $\set{b_k^j}_{k,j}\subset \R$ 
      \eq{
	\norm{y(t)-\yref(t)}\to 0
      } 
      at a uniform exponential rate as $t\to\infty$.
  \end{itemize}
\end{ORP}

As shown in~\cite{PauPha20}, the controller constructed in this paper is also \keyterm{robust} in the sense that it tolerates changes and uncertainty in some of the parameters of the system~\eqref{eq:EBmodel} (see~\cite{PauPha20} for details).

\section{Reduced Order Internal Model Based Controller Design}

The controller we design is the finite-dimensional ``Observer-based robust controller'' introduced in~\citel{PauPha20}{Sec.~III.A}. It has the general form
\begin{subequations}
  \label{eq:FinConObs}
  \eqn{
    \hspace{-1ex} \dot{z}_1(t)&= G_1z_1(t) + G_2 e(t)\\
  \label{eq:FinConObs2}
    \hspace{-1ex}\dot{z}_2(t)&= (A_L^r+B_L^rK_2^r)z_2(t) + B_L^r K_1^N z_1(t) -L^r e(t)\\
    \hspace{-1ex}u(t)&= K_1^N z_1(t) + K_2^rz_2(t) 
  }
\end{subequations}
with state $(z_1(t),z_2(t))^T\in Z:= Z_0\times \C^r$ and input $e(t)=y(t)-\yref(t)$.
The matrices $(G_1,G_2,A_L^r,B_L^r,K_1^N,K_2^r,L^r)$ 
are
constructed using
 the algorithm in Section~\ref{sec:ContrAlgorithm}.
Theorem III.1 in~\cite{PauPha20}
 states that if the order $N\in\N$ of the Galerkin approximation in \textbf{Step~2} of the algorithm and the order $r\leq N$ of the Balanced Truncation model reduction in \textbf{Step~4} are sufficiently high, then the controller~\eqref{eq:FinConObs} solves the output regulation problem.

\subsection{The Galerkin Approximation of the Beam Model}
\label{sec:GalerkinApprox}

  The controller design uses a Galerkin approximation $(A^N,B^N,C^N)$ of~\eqref{eq:EBmodel} written as first order system
  \eq{
    \dot{x}(t)&= Ax(t)+Bu(t)+B_d \wdist(t)\\
    y(t)&= Cx(t)
  }
  on a Hilbert space $X$.
We review the construction of $(A^N,B^N,C^N)$ in this section.
  As described in~\cite{PauPha20}, the controller design algorithm requires that the operator $A$ is associated to a coercive sesquilinear $a(\cdot,\cdot)$ form defined on another Hilbert space $V\subset X$. 
  In order to utilise the spectral Galerkin method based on Chebyshev polynomials, we formulate the beam system~\eqref{eq:EBmodel} as a first order system with state $x(t)=(v(\cdot,t),\dot{v}(\cdot,t))^T$ on the space $X=V_0\times \Lp[2]_\gw(-1,1)$ where $\Lp[2]_\gw(-1,1)$ is the $\Lp[2]$-space with the weight $\gw(\xi)=(1-\xi^2)^{-1/2}$ and $V_0=\setm{f\in H_\gw^2(-1,1)}{f(\pm 1)=f'(\pm 1)=0}$. Here also the Sobolev space $H^2_\gw(-1,1)$ is defined with the weight $\gw(\cdot)$. The motivation for the use of the weighted spaces is that the Chebyshev polynomials $T_k(\cdot)$ are orthogonal with respect to the inner product of $\Lp[2]_\gw(-1,1)$, and this property can be leveraged in computing the Galerkin approximation~\cite{She95}.

  Similarly as in~\citel{ItoMor98}{Sec.~3}, letting $\wdist(t)\equiv 0$, taking an inner product $\iprod{\cdot}{\cdot}_\gw$ (the inner product of $\Lp[2]_\gw(-1,1)$) of both sides of~\eqref{eq:EBmodel} with a $\psi_1\in V_0$ and integrating by parts leads to the weak form
  \begin{subequations}
    \label{eq:weakform2ndorder}
    \eqn{
      \MoveEqLeft\iprod{\ddot{v}}{\psi_1}_\gw
      + \iprod{EIv''+d_{KV}I \dot{v}''}{(\gw\psi_1)''}_{\Lp[2]} 
 + d_v\iprod{\dot{v}}{\psi_1}_\gw \\
&= \iprod{b_1u_1+b_2u_2}{\psi_1}_\gw.
    }
  \end{subequations}
  Defining the state of the first order system as $x(t)=(v(\cdot,t),\dot{v}(\cdot,t))$, the above weak form can be written as
  \eq{
    \iprod{\dot{x}(t)}{\psi}_X + a(x(t),\psi)=\iprod{Bu(t)}{\psi}_X, \quad \psi\in V
  }
  on $V=V_0\times V_0$ if we define $Bu= \pmatsmall{0\\ b_1u_1+b_2u_2} $ for $u=(u_1,u_2)^T\in\C^2$ and
  \eq{
    a(\phi,\psi) &= -\iprod{\phi_2}{\psi_1}_{V_0} + \iprod{EI \phi_1''+d_{KV}I \phi_2''}{(\gw \psi_2)''}_{\Lp[2]} 
 + d_v \iprod{\phi_2}{\psi_2}_\gw  
  }
 for $\phi=(\phi_1,\phi_2)^T\in V$, $\psi=(\psi_1,\psi_2)^T\in V$.
By~\citel{Mad90}{Lem.~5.1}
we can define an inner product on $V_0$ by
\eq{
\iprod{\phi_1}{\psi_1}_{V_0} = EI \iprod{\phi_1''}{(\gw \psi_1)''}_{\Lp[2]}, \qquad \phi_1,\psi_1\in V_0,
}
and the norm $\norm{\cdot}_{V_0}$ induced by $\iprod{\cdot}{\cdot}_{V_0}$ is equivalent to the norm on $H_\gw^2(-1,1)$.
The following lemma shows that if $X=V_0\times \Lp[2]_\gw(-1,1)$ and $V=V_0\times V_0$ are equipped with the norms 
  $\norm{(\phi_1,\phi_2)^T}_X^2 = \norm{\phi_1}_{V_0}^2 + \norm{\phi_2}_\gw^2$ and
  $\norm{(\phi_1,\phi_2)^T}_V^2 = \norm{\phi_1}_{V_0}^2 + \norm{\phi_2}_{V_0}^2$, 
then the form $a(\cdot,\cdot)$ is bounded and coercive, and therefore the standing assumptions   in~\cite{PauPha20} are satisfied.
The proof of Lemma~\ref{lem:FormBddCoercive} using~\cite{Mad90} was
given to the authors by 
 the anonymous referee.
\begin{lemma}
\label{lem:FormBddCoercive}
The sesquilinear form 
$a(\cdot,\cdot)$ is bounded and coercive, i.e., there exist $q_1,q_2>0$ and $\gl_0\in\R$ such that
  \eq{
    \abs{a(\phi,\psi)}&\leq q_1 \norm{\phi}_V \norm{\psi}_V, \qquad \qquad \forall \phi,\psi\in V\\
    \re a(\phi,\phi)&\geq q_2 \norm{\phi}^2_V - \gl_0 \norm{\phi}^2_X ~\qquad \forall \phi\in V.
  }
\end{lemma}

\begin{proof}
By~\citel{Mad90}{Lem.~5.1} 
there exists 
$\gb>0$ such that $\norm{f}_\gw 
\leq \gb \norm{f}_{V_0}$
 for all $f\in V_0$.
Let
$\phi=(\phi_1,\phi_2)^T\in V$ and $\psi=(\psi_1,\psi_2)^T\in V$ be arbitrary. 
Since $\abs{\iprod{\phi_k}{\psi_j}_{V_0}}\leq \norm{\phi_k}_{V_0}\norm{\psi_j}_{V_0}\leq \norm{\phi}_V\norm{\psi}_V$ for $k,j\in \set{1,2}$, we have
\eq{
    \abs{a(\phi,\psi)} 
&= \bigl|-\iprod{\phi_2}{\psi_1}_{V_0} 
+ \iprod{ \phi_1}{\psi_2}_{V_0} 
+\frac{d_{KV}}{E} \iprod{\phi_2}{\psi_2}_{V_0} 
   + d_v \iprod{\phi_2}{\psi_2}_\gw \bigr| \\
&\leq \left( 2+\frac{d_{KV}}{E} \right) \norm{\phi}_V \norm{\psi}_V + d_v \norm{\phi_2}_\gw \norm{\psi_2}_\gw\\
&\leq \left( 2+\frac{d_{KV}}{E}  + d_v\gb^2\right) \norm{\phi}_V \norm{\psi}_V .
}
Thus the first claim holds with $q_1=2+d_{KV}/E+d_v\gb^2$. 
Moreover, the definitions of $\norm{\cdot}_V$ and $\norm{\cdot}_X$ imply
\eq{
    \re a(\phi,\phi)
&=  \frac{d_{KV}}{E} \iprod{\phi_2}{\phi_2}_{V_0} + d_v \iprod{\phi_2}{\phi_2}_\gw \\
&=  \frac{d_{KV}}{E} \norm{\phi}_V^2 - \frac{d_{KV}}{E} \norm{\phi_1}_{V_0}^2 + d_v \norm{\phi_2}_\gw^2\\
&\geq  \frac{d_{KV}}{E} \norm{\phi}_V^2 - \frac{d_{KV}}{E} \left(\norm{\phi_1}_{V_0}^2 +  \norm{\phi_2}_\gw^2\right)
}
and thus the second claim holds with $q_2=\gl_0=\frac{d_{KV}}{E}$.
\end{proof}

  The Galerkin approximation is defined by constructing a sequence of approximating finite-dimensional subspaces $V^N$ of $V$. The approximating subspaces are required to have the property that (see~\citel{Mor94}{Sec.~5.2}) 
\textit{any element $\phi\in V$ can be approximated by elements in $V^N$ in the norm on $V$, i.e.,
}
\eq{
  \forall \phi \in V\,\exists (\phi^N)_N, \,\phi^N\in V^N: \quad \norm{\phi^N-\phi}_V \stackrel{N\to\infty}{\longrightarrow} 0 .
}
The matrix $A^N:~V^N \to V^N$ is defined 
by restricting $a(\cdot,\cdot)$ to $V^N \times V^N$, i.e., 
  \eq{
    \langle -A^N \phi, \psi  \rangle = a (\phi, \psi) \quad \text{for all} \quad \phi, \psi \in V^N.
  }
  The input matrix $B^N \in \Lin(U, V^N )$ is defined by 
  \eq{
    \langle B^N u, \psi  \rangle = \iprod{u}{ B^* \psi } \quad \text{for all} \quad
    \psi \in V^N, 
  }
  and $C^N \in \Lin(V^N, Y)$ is the restriction of $C\in \Lin(X,Y)$ onto $V^N$. 
  Approximating (or even knowing!) $B_d\in \Lin(U_d,X)$ is not necessary in our controller design.

  In this paper we define the approximating subspaces $V^N=V_0^N\times V_0^N$ according to the \keyterm{spectral Galerkin method} in~\cite{She95}. 
  The basis functions $\phi_k$ of $V^N$ are defined to be linear combinations of the Chebyshev polynomials $T_k(\xi)$ 
  so that $\phi_k$
  satisfy the boundary conditions of $V=V_0\times V_0$.
  As shown in~\citel{She95}{Sec.~3.1}, 
  we can choose 
  \eq{
    V_0^N= \Span \set{\phi_0(\cdot),\ldots,\phi_{N-4}(\cdot)}
  }
  where
  for $k=\List[0]{N-4}$ we have
  \eq{
    \phi_k(\xi) = T_k(\xi) - \frac{2(k+2)}{k+3}T_{k+2}(\xi) + \frac{k+1}{k+3}T_{k+4}(\xi) .
  }
  Then $V_0^N=\setm{f\in\Span \set{T_0(\cdot),\ldots,T_N(\cdot)}}{f(\pm 1)=f' (\pm1)=0}$.
In the Galerkin approximation the solution $v(\xi,t)$ of~\eqref{eq:EBmodel} is approximated with 
\eq{
  v_N(\xi,t)=\sum_{k=0}^{N-4}\ga_k(t)\phi_k(\xi).
}
The matrices of the approximate system $(A_N,B_N,C_N)$ are then derived from the system of ordinary differential equations which are obtained from the second order weak form~\eqref{eq:weakform2ndorder} with $\psi_2=\phi_l$ for $l\in \List[0]{N-4}$.
The resulting system is
\eq{
  \MoveEqLeft\sum_{k=0}^{N-4}
  \iprod{\phi_k}{\phi_l}_\gw (\ddot{\ga}_k(t)+d_v\dot{\ga}_k(t))\\
    &\quad +
  \sum_{k=0}^{N-4}
  \iprod{\phi_k''}{(\gw\phi_l)''}_{\Lp[2]} (EI\ga_k(t)+d_{KV}I\dot{\ga}_k(t))\\
    &= \iprod{b_1}{\phi_l}_\gw u_1(t)
    + \iprod{b_2}{\phi_l}_\gw u_2(t)
}
for $l\in \List[0]{N-4}$.
The output matrix $C_N$ defined by
\eq{
  y(t)
  =\pmat{v_N(\xi_1,t)\\v_N(\xi_2,t)}
 = \sum_{k=0}^{N-4} \pmat{\phi_k(\xi_1)\\\phi_k(\xi_2)}\ga_k(t).
}
Setting $\bm{\ga}(t)=(\ga_0(t),\ldots,\ga_{N-4}(t))^T$, we have 
\eq{
\MoveEqLeft M \ddot{\bm{\ga}}(t) + EI\cdot F \bm{\ga}(t) + (d_{KV} F+d_v M)\dot{\bm{\ga}}(t) = B_0^N u(t)\\
y(t)&=C_0^N \bm{\ga}(t),
}
where $M = (\iprod{\phi_k}{\phi_l}_\gw)_{lk}\in \R^{(N-3)\times (N-3)}$ and $F=(\iprod{\phi_k''}{(\gw\phi_l)''}_{\Lp[2]})_{lk}\in \R^{(N-3)\times (N-3)}$.
The exact values of the inner products 
$M_{lk}:=\iprod{\phi_k}{\phi_l}_\gw$ and
$F_{lk}:=\iprod{\phi_k''}{(\gw\phi_l)''}_{\Lp[2]}$ 
are given
in Lemma~\ref{lem:GalBasisMatrices} below.
The values 
$\iprod{b_1}{\phi_l}_\gw $ and
$\iprod{b_2}{\phi_l}_\gw $ can be computed
based on the (truncated) Chebyshev series expansions of $b_1$ and $b_2$. Indeed, if $b_j(\xi)=\sum_{k=0}^\infty q_k^jT_k(\xi)$, then the orthogonality of the basis $\set{T_k}_{k=0}^\infty$ with respect to the inner product $\iprod{\cdot}{\cdot}_\gw$ implies that for $j=1,2$ and $l\in \List[0]{N-4}$ we have
\eq{
  \iprod{b_j}{\phi_l}_\gw 
  &=\iprod{b_j}{T_l}_\gw 
  -\frac{2(l+2)}{l+3}\iprod{b_j}{T_{l+2}}_\gw 
  +\frac{l+1}{l+3}\iprod{b_j}{T_{l+4}}_\gw \\
  &=\norm{T_l}_\gw^2 q_l^j
  -\frac{2(l+2)}{l+3}\norm{T_{l+2}}_\gw^2 q_{l+2}^j 
  +\frac{l+1}{l+3}\norm{T_{l+4}}_\gw^2 q_{l+4}^j .
}
Here 
$\norm{T_0}_\gw^2 = \pi$
and $\norm{T_l}_\gw^2 = \pi/2$ for $l\geq 1$.
In Matlab the required Chebyshev coefficients $\set{q_l^j}_{l=0}^N$, $j=1,2$, of $b_1$ and $b_2$ are readily available using \textbf{Chebfun}.
Finally, based on the formula for $\phi_k(\cdot)$, the matrix $C_0^N$ has the form $C_0^N=\pmat{\bm{c}_0,\ldots,\bm{c}_{N-4}}$ where 
\eq{
  \bm{c}_k =
  \pmat{
    T_k(\xi_1)\\T_k(\xi_2)}
- \frac{2(k+2)}{k+3}
 \pmat{
   T_{k+2}(\xi_1)\\
   T_{k+2}(\xi_2)
 }
 + \frac{k+1}{k+3}
 \pmat{
   T_{k+4}(\xi_1)\\
   T_{k+4}(\xi_2)
 }.
}

\begin{lemma}\textup{\citel{She95}{Lem.~3.1}}
  \label{lem:GalBasisMatrices}
  The nonzero components $M_{lk}$,
  $l,k\in \List[0]{N-4}$,
  of $M\in \R^{(N-3)\times (N-3)}$
  are $M_{00}=35\pi /18$,
  \eq{
M_{ll} &= \frac{\pi\left[(l+1)^2+4(l+2)^2+(l+3)^2\right]}{2(l+3)^2}, \quad l\geq 1\\
M_{l,l+2}&=M_{l+2,l} = -  \frac{\pi\left[(l+2)(l+5)+(l+1)(l+4)\right]}{(l+3)(l+5)} \\
    M_{l,l+4}&=M_{l+4,l} = \frac{\pi(l+1)}{2(l+3)}.
  }
  The nonzero elements $F_{lk}$,
  $l,k\in \List[0]{N-4}$,
  of $F\in \R^{(N-3)\times (N-3)}$ are
\eq{
    F_{ll}&= 8(l+1)^2(l+2)(l+4)\pi\\
    F_{lk}&= \frac{8\pi (l+1)(l+2)\left[l(l+4)+3(k+2)^2  \right]}{k+3}, 
  }
    for $ k=l+2,l+4,\ldots$.
\end{lemma}

In summary, the matrices $(A_N,B_N,C_N)$ of the Galerkin approximation are given by
$C_N=\pmat{C_0^N,~0}$ and
\eq{
  A_N = \pmat{0&I_{N\times N}\\-EI M\inv F& \hspace{.2cm}-d_{KV}M\inv F - d_vI \cdot I_{N\times N}}, \quad
  B_N = \pmat{0\\M\inv B_0^N}.
}

  \subsection{The Reduced Order Controller Design Algorithm}
  \label{sec:ContrAlgorithm}

  The following algorithm from~\citel{PauPha20}{Sec.~III.A} determines the parameters of the controller~\eqref{eq:FinConObs} so that 
  $(G_1,G_2)$ are as in \textbf{Step~1}, $K_1^N$ is as in \textbf{Step~3}, and $(A_L^r,B_L^r,L^r,K_2^N)$ are as in \textbf{Step~4}.  
The parts $G_1,G_2,K_1^N$ constitute \keyterm{the internal model} of the controller, and its construction is based on the fact that the system~\eqref{eq:EBmodel} has two outputs, i.e., $Y=\C^2$.

\smallskip

\noindent\textbf{PART I. The Internal Model}

\smallskip

\noindent \textbf{Step 1:}
Choose $Z_0 = Y^{2q+1} = \C^{4q+2}$,
  \eq{
    G_1 &= \diag(0_2, \Omega_1,\ldots, \Omega_q)\in \Lin(Z_0)\\
    G_2&= \pmat{I_2,I_2,0_2,\ldots,I_2,0_2}^T\in \Lin(Y,Z_0),
  }
  where
  $\Omega_k = \pmatsmall{0_2&\gw_k I_2\\-\gw_k I_2&0_2}$ for all $k\in \List{q}$ and
  $0_2$ and $I_2$ are the $2\times 2$ zero and identity matrices.
The pair $(G_1,G_2)$ is controllable by construction.

\smallskip

\noindent \textbf{PART II. The Galerkin Approximation and Stabilization}.

\smallskip

\noindent \textbf{Step 2:}
For a sufficiently large $N\in\N$, form a Galerkin approximation  $(A^N,B^N,C^N)$  on $V^N$ of the control system~\eqref{eq:EBmodel} as described 
in Section~\ref{sec:GalerkinApprox}.

\smallskip
\noindent \textbf{Step 3:}
Choose parameters $\ga_1,\ga_2\geq 0$, $Q_1\in \Lin(U_0,X)$, and $Q_2\in \Lin(X,Y_0)$ with $U_0,Y_0$ Hilbert spaces in such a way that the systems
$(A + \ga_1 I,Q_1,C)$  and $(A+\ga_2 I,B,Q_2)$
are exponentially stabilizable and detectable.
Let $Q_1^N$ and $Q_2^N$ be the approximations of $Q_1$ and $Q_2$, respectively, according to the approximation $V^N$ of $V$.
Let $Q_0\in \Lin(Z_0,\C^{p_0})$ be such that $(Q_0,G_1)$ is observable, and let $R_1\in \Lin(Y)$ and $R_2\in \Lin(U)$ be positive definite matrices.
Denote
\eq{
  \Acomp^N = \pmat{G_1&G_2C^N\\0&A^N}, \; \Bcomp^N=\pmat{G_2D\\B^N}, \; \Qcomp^N = \pmat{Q_0&0\\0&Q_2^N}.
}
Define
\eq{
L^N &=-\Sigma_N C^N R_1\inv\in \Lin(Y, V^N) \\
K^N &= \pmat{K_1^N,\; K_2^N} =-R_2\inv (\Bcomp^N)^\ast\Pi_N \in \Lin(Z_0\times V^N ,U )
}
where
$\Sigma_N$
and
$\Pi_N$
are the non-negative solutions of the finite-dimensional Riccati equations
\eq{
  (A^N + \ga_1 I) \Sigma_N + \Sigma_N (A^N + \ga_1 I)^* 
  - \Sigma_N \left(C^N \right)^* R_1\inv C^N \Sigma_N &=- Q_1^N (Q_1^N)^*  \\
  (\Acomp^N + \ga_2 I)^* \Pi_N + \Pi_N (\Acomp^N +\ga_2 I) 
  - \Pi_N \Bcomp^N R_2\inv\left(\Bcomp^N\right)^\ast \Pi_N &=- 
  \left(\Qcomp^N\right)^\ast \Qcomp^N.  
}
With these choices $\Acomp^N+\Bcomp^NK^N$ and $A^N+L^NC^N$ are Hurwitz if $N$ is sufficiently large~\citel{BanIto97}{Thm. 4.8}.

\smallskip

\noindent\textbf{PART III.} \textbf{The Model Reduction} 

\smallskip

\noindent \textbf{Step 4:}
  For a fixed and suitably large $r\in\N$, $r\leq N$,
  apply the Balanced Truncation method to the stable finite-dimensional system
  \eq{
    (A^N+L^NC^N, [ B^N+L^N D,\;L^N],K^N_2)
  }
to obtain a stable $r$-dimensional reduced order system
\eq{
  \left(A_L^r,[B_L^r ,\; L^r] ,K_2^r\right).
}

\section{Numerical Simulations}

The simulation codes can be downloaded from GitHub at the address {\url{https://github.com/lassipau/MTNS20-Matlab-simulations}}.\\
The controller is designed for a beam~\eqref{eq:EBmodel} with parameters
\eq{
  E = 10, \quad I = 1, \quad d_{KV} = 0.01, \quad d_v = 0.4.
}
The pointwise measurements of the defection are at $\xi_1 = -0.6$ and $\xi_2 = 0.3$.
The profile functions for the control $u(t) = (u_1(t),u_2(t))^T$ and disturbance $\wdist(t)\in\R$ are 
\eq{
  b_1(\xi) &= \frac{1}{3}(\xi+1)^2(1-\xi)^6, \quad~
  b_2(\xi) = \frac{1}{3}(\xi+1)^6(1-\xi)^2\\
  b_d(\xi) &= \frac{1}{3}(\xi+1)^2(1-\xi)^2.
}

Our goal is to design a controller which is capable of tracking and rejecting continuous $2$-periodic signals $\yref(t)$ and $\wdist(t)$. To this end, we will choose the frequencies $0=\gw_0<\gw_1<\ldots<\gw_q$ to be $\gw_k=k\pi$ for $k\in \List[0]{q}$ for some suitable value $q\in\N$, and we use $q=10$ in our simulations. Due to the robustness of the controller, for any continuous $2$-periodic reference signal $\yref^{per}(t)$ the controller will track the truncated part $\yref(t)$ of the form~\eqref{eq:yrefwdist} with perfect accuracy, and the remaining part $\yref(t)-\yref^{per}(t)$ will appear as an additional external disturbance in the control loop. However, due to the stability of the closed-loop system, the effect of this difference signal on the asymptotic regulation error $e(t)$ is guaranteed to be small as long as the quantity $\max_{t\in[0,2]}\norm{\yref(t)-\yref^{per}(t)}$ is sufficiently small. 
The smallness of the truncation error can be guaranteed for the most important continuous $2$-periodic functions $\yref^{per}(\cdot)$ (outside pathological situations) by choosing a suitably large $q\in\N$. In the simulations we demonstrate the above property of the controller in the tracking of a triangle signal which contains an infinite number of frequency components.

The spectral Galerkin approximation used in the controller design is completed with $\dim V_0^N=39$, in which case the dimension of the approximate first order system $(A^N,B^N,C^N)$ is $78$.
The uncontrolled system is exponentially stable and it has a stability margin which is approximately $0.35$. 
In \textbf{Step~3} of controller design we choose the parameters as $\ga_1=2$, $\ga_2=0.8$ and $R_1=R_2=I$, $Q_0=I$, $Q_1=I$ and $Q_2=I$.  
The order $r\leq N$ of the model reduction in \textbf{Step~4} of the controller design can be chosen to be as low as $r=4$ while 
still achieving exponential closed-loop stability. Using such a low-dimensional ``observer-part'' in the controller is made possible by the fact that the uncontrolled system already has quite strong stability properties.
The resulting closed-loop system has a stability margin 
$\approx 1.01$.

Since the internal model in the controller has dimension $2(2q+1)=42$, the full size of the controller is $\dim Z=46$.
For simulations we approximate the controlled beam system~\eqref{eq:EBmodel} with a separate higher-dimensional spectral Galerkin approximation with $\dim V_0^N=69$.
Figure~\ref{fig:ROMCLeigvals} plots some of the eigenvalues of the closed-loop system.

\begin{figure}[h!]
  \centering
  \includegraphics[width=0.7\linewidth]{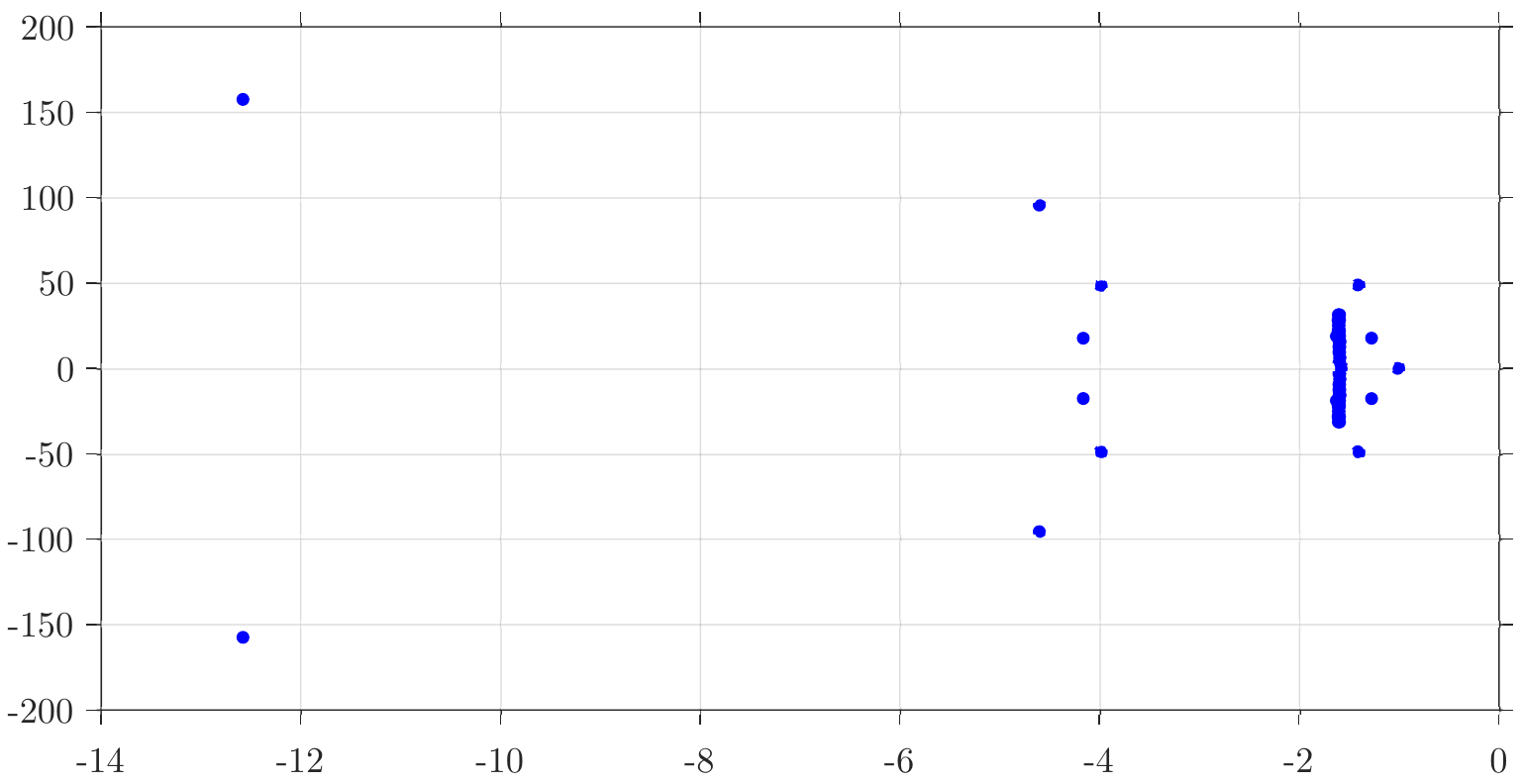}
  \caption{Eigenvalues of the closed-loop system.}
  \label{fig:ROMCLeigvals}
\end{figure}

The controller design algorithm in Section~\ref{sec:ContrAlgorithm} requires that the system does not have transmission zeros at the (complex) frequencies $\set{i\gw_k}_{k=0}^q\subset i\R$ of the reference and disturbance signals. This property can be tested numerically using the Galerkin approximation of~\eqref{eq:EBmodel} and it is satisfied in our simulations.

The designed internal model based controller is capable of tracking any reference signal $\yref(t)$ and any disturbance signal $\wdist(t)$ of the form~\eqref{eq:yrefwdist} with frequencies $(\gw_k)_{k=0}^{10}$ with arbitrary unknown amplitudes and phases. As explained above, the controller will also \keyterm{approximately} track and reject any continuous $2$-periodic reference and disturbance signals. Figure~\ref{fig:ROMoutput} shows the output $y(t)$ and the tracking error in the case where the first component of $\yref(t)$ is a $2$-periodic triangle signal, the second component of $\yref(t)$ is identically zero, and the disturbance signal is $\wdist(t)=\sin(\pi t)+0.4\cos(3\pi t)$. 
The initial deflection $v_0(\xi)$, the initial velocity $v_1(\xi)$ and the initial state of the controller are all zero.

\begin{figure}[h!]
  \centering
  \includegraphics[width=0.65\linewidth]{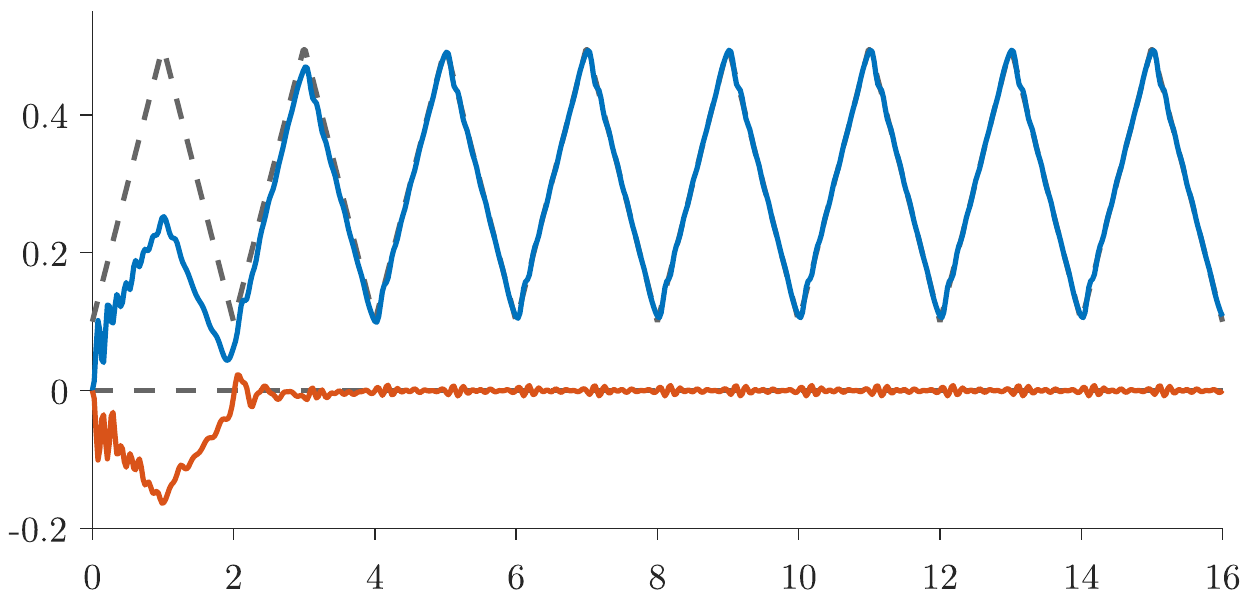}
  \caption{The output $y(t)$ and the reference signal $\yref(t)$.}
  \label{fig:ROMoutput}
\end{figure}

Figure~\ref{fig:ROMerror} plots the norm of the tracking error $\norm{e(t)}=\norm{y(t)-\yref(t)}_{\R^2}$. The asymptotic residual error is due to the fact that the internal model contains only a finite number of frequencies and is therefore not capable of tracking the nonsmooth triangle signal with perfect accuracy.

\begin{figure}[h!]
  \centering
  \includegraphics[width=0.65\linewidth]{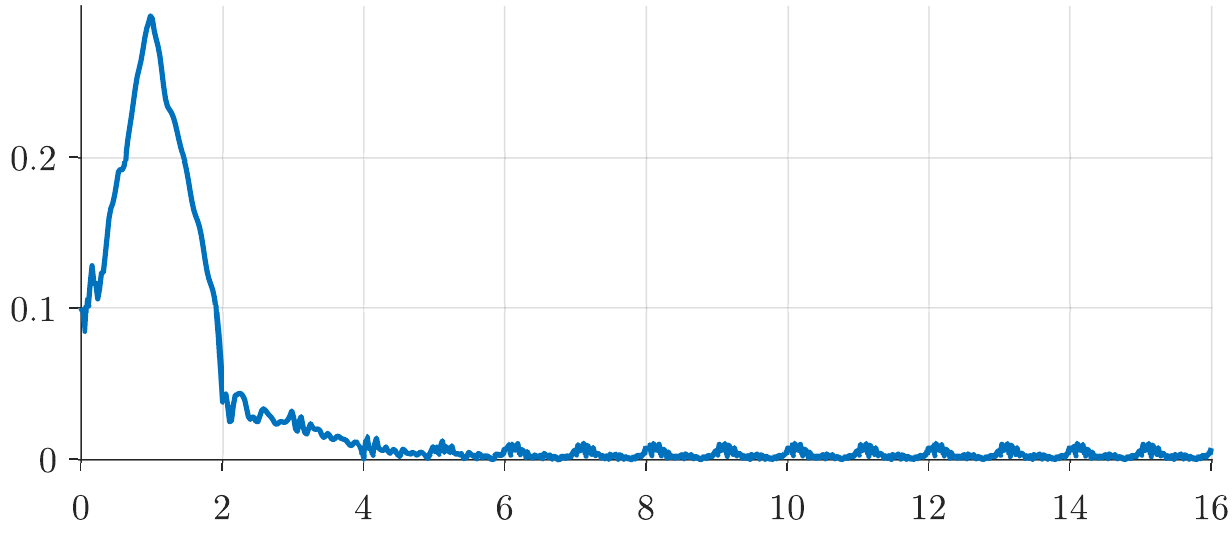}
  \caption{The norm $\norm{e(t)}$ of the tracking error.}
  \label{fig:ROMerror}
\end{figure}

Figure~\ref{fig:ROMcontrols} depicts the control actions $u_1(t)$ and $u_2(t)$ in the simulation. Finally, Figure~\ref{fig:ROMstate} plots the deflection $v(\xi,t)$ of the controlled beam.

\begin{figure}[h!]
  \centering
  \includegraphics[width=0.65\linewidth]{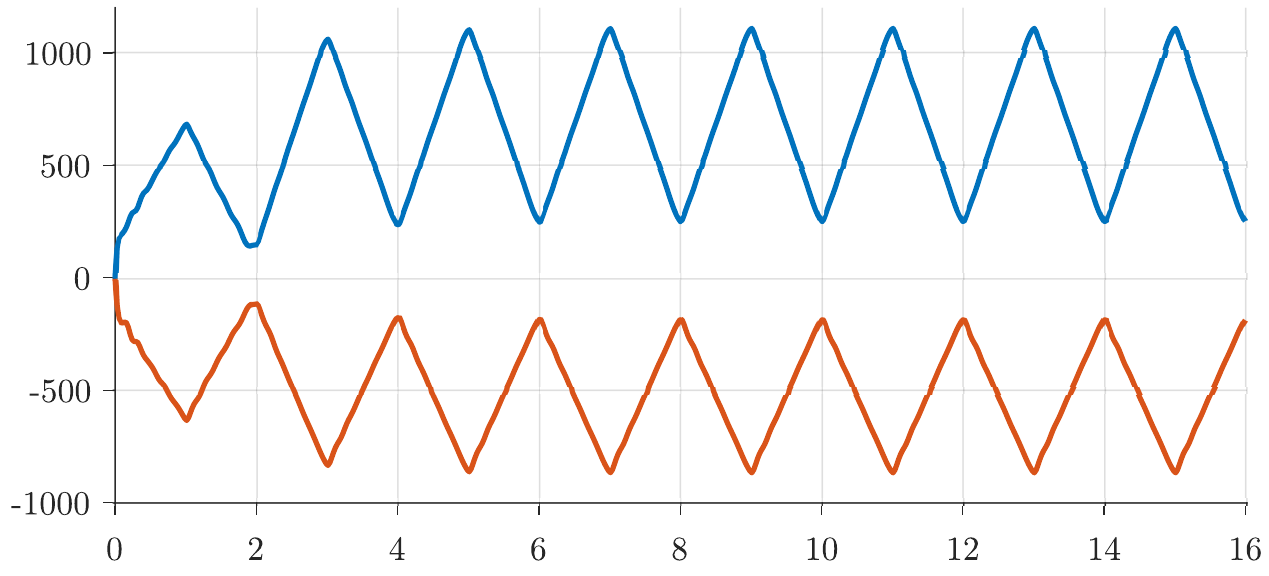}
  \caption{The control inputs $u_1(t)$ (blue) and $u_2(t)$ (red).}
  \label{fig:ROMcontrols}
\end{figure}

\begin{figure}[h!]
  \centering
  \includegraphics[width=0.8\linewidth]{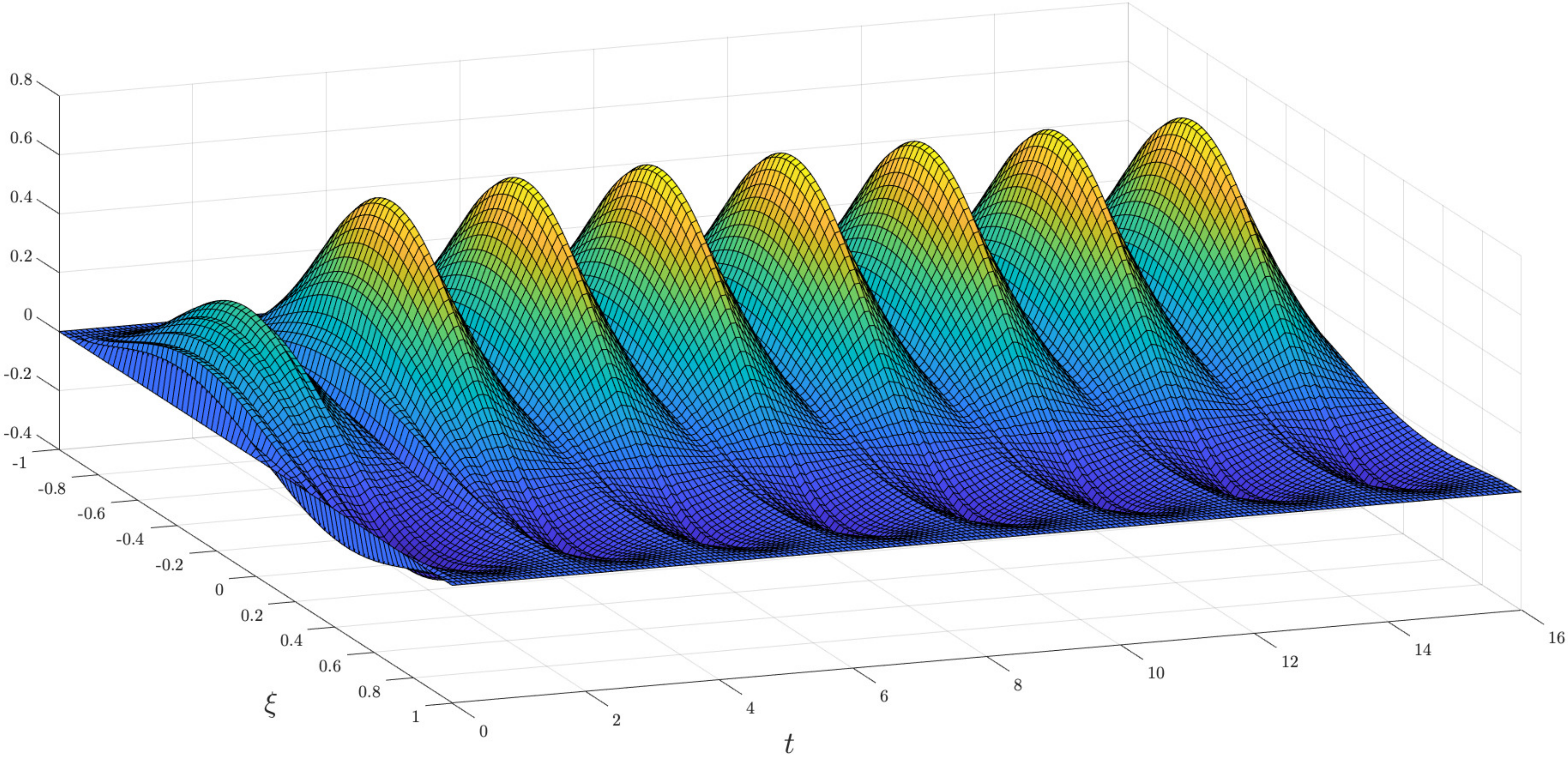}
  \caption{The deflection of the controlled beam.}
  \label{fig:ROMstate}
\end{figure}

\subsection{Comparison with a Low-Gain Controller}

Because the beam system is exponentially stable, the output tracking problem can alternatively be solved with a ``simple'' internal model based controller structure whose dynamics only contain the internal model, i.e., \eq{
  \dot{z}(t)&=G_1z(t)+G_2e(t), \qquad z(0)\in \R^{n_c},\\
  u(t)&=Kz(t),
}
where $G_1$ and $G_2$ are as in Section~\ref{sec:ContrAlgorithm}, and thus $n_c=2(2q+1)=42$. As shown in~\cite{HamPoh00,RebWei03} or~\citel{Pau16a}{Rem.~10}, 
 the matrix $K\in \R^{2\times n_c}$ can be chosen as 
\eq{
  K&=\eps \bigl[P(0)\inv,\re P(i\gw_1)\inv,\im P(i\gw_1)\inv,\ldots,
   \re P(i\gw_q)\inv,\im P(i\gw_q)\inv\bigr],
}
where $P(\gl)\in \C^{2\times 2}$ is the transfer function of the system~\eqref{eq:EBmodel}.
For this choice of $K$, there exists $\eps^\ast >0$ such that for any $0<\eps\leq \eps^\ast$ the controller achieves output tracking and disturbance rejection for the beam system. 
In the controller design the values $P(i\gw_k)$ can be reliably computed using a Galerkin approximation of the system due to the results in~\cite{Mor94}.

The low-gain parameter $\eps>0$ should be designed so that the closed-loop has the best possible stability margin, and this can be achieved using root locus type analysis. For the beam system~\eqref{eq:EBmodel} with the chosen parameters and the internal model $(G_1,G_2)$ with frequencies $\gw_k=k\pi$ for $k\in \List[0]{10}$, 
the stabilization of the 
closed-loop system
arising from this controller structure
 is very difficult  
 due to the high number of unstable frequencies in the internal model.
If the number of frequencies is reduced to $\gw_k=k\pi$ for $k\in \List[0]{5}$ (potentially resulting in lower accuracy of the tracking), 
the best achievable closed-loop stability margin is approximately
 $0.0382$
 with parameter $\eps=0.076$. The margin can be compared to the stability margin 
$\approx 1.01$
 achieved with the reduced order controller in Section~\ref{sec:ContrAlgorithm}. This smaller stability margin leads to a significantly slower rate of convergence of the output. Figure~\ref{fig:LGoutput} depicts the output of the controlled system with the low-gain controller with $\eps=0.076$.

\begin{figure}[h!]
  \centering
  \includegraphics[width=0.7\linewidth]{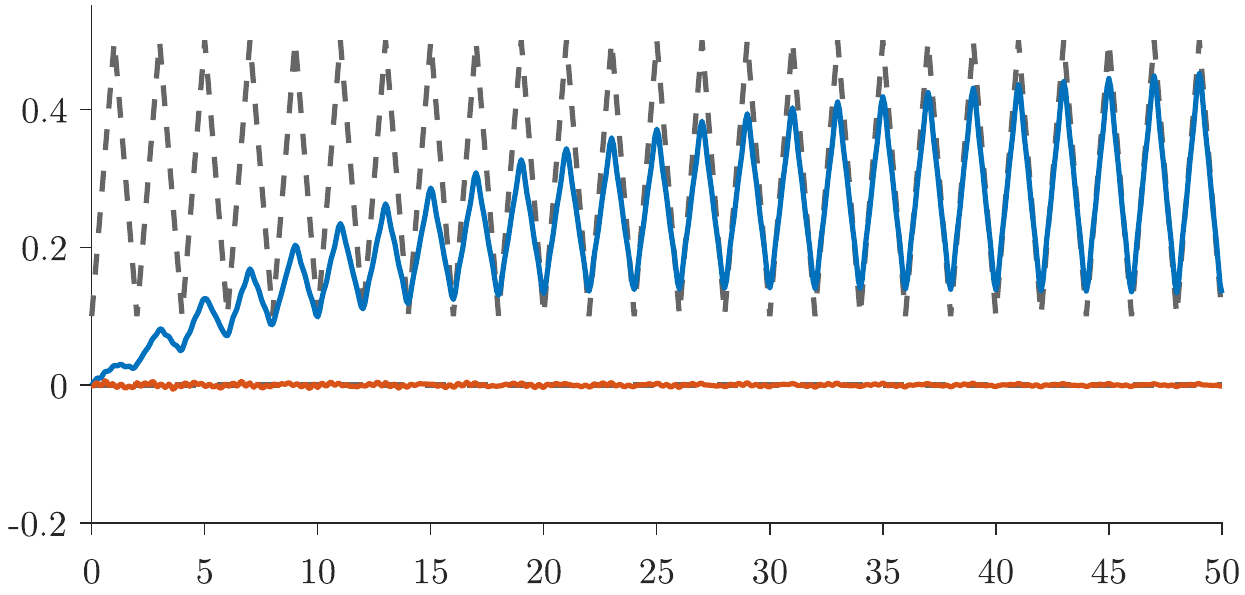}
  \caption{Output of the system with the low-gain controller.}
  \label{fig:LGoutput}
\end{figure}

Note that since in our main controller we were able to choose the size of the reduced order model in \textbf{Step~4} as $r=4$, the 
dimension of our reduced order controller is only higher by $4$ compared to the low-gain controller, but this additional structure of the controller allows the design of a significantly larger stability margin. 
Since the second controller is based on a low-gain design, it is natural to ask 
if using this controller leads to reduced control action.
However,
 in the case of a stable linear system with the same number of inputs and outputs
 the control action required to produce the desired asymptotic output $\yref(t)$
 does not 
 depend on the type of the internal model based controller
(see \citel{Pau17b}{Lem.~3.10 and proofs of Thm.~3.3 and Lem.~3.4}).
Therefore
 in our simulation 
the low-gain controller does not achieve 
output regulation with smaller control gains and 
the controller requires
 less control action only for small values of $t\geq 0$.  
The control inputs produced by the low-gain controller are depicted in Figure~\ref{fig:LGcontrols}.

\begin{figure}[h!]
  \centering
  \includegraphics[width=0.7\linewidth]{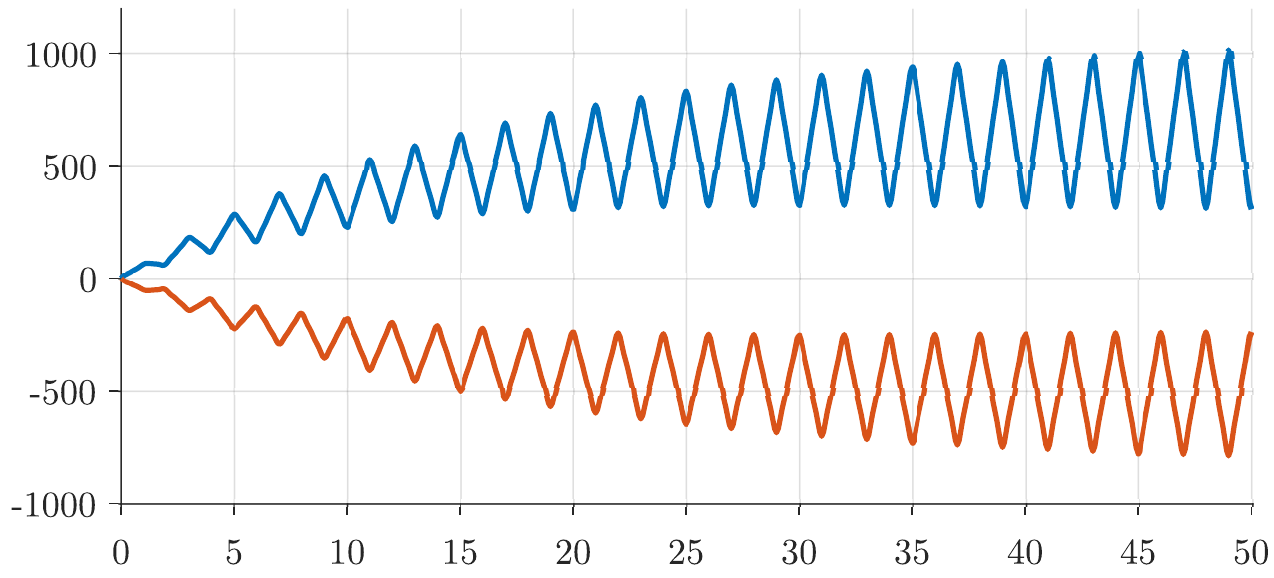}
  \caption{Control inputs $u_1(t)$ (blue) and $u_2(t)$ (red) produced by the low-gain controller.}
  \label{fig:LGcontrols}
\end{figure}

\end{document}